\documentclass[12pt,reqno]{amsart}
\usepackage{graphicx}
\usepackage{epstopdf}
\usepackage{url}
\usepackage{latexsym}
\usepackage{amsfonts,amsmath,amssymb}
\newtheorem{theorem}{Theorem}[section]
\newtheorem{thm}{Theorem}[section]

\newtheorem{lem}[theorem]{Lemma}
\newtheorem{definition}[theorem]{Definition}

\newtheorem{remark}{Remark}

\newtheorem{fact}[theorem]{Fact}
\newtheorem{observation}[theorem]{Observation}

\def\S{{\mathfrak S}}
\def\Z{\mathbb{Z}}
\def\N{\mathbb{N}}

\def\lpk{\mathrm{lpk}\,}
\def\rise{\textrm{asc}\,}

\def\des{\textrm{des}\,}
\def\odes{\textrm{des}_1\,}
\def\edes{\textrm{des}_0\,}
\def\erise{\textrm{asc}_0\,}
\def\orise{\textrm{asc}_1\,}
\def\inv{\textrm{inv}\,}
\def\iasc{\textrm{asc}_i\,}
\def\ides{\textrm{des}_i\,}

\def\B{\mathcal{B}}
\def\D{\mathrm{Des}}
\def\oD{\mathrm{Des}_1}
\def\eD{\mathrm{Des}_0}
\def\od{\mathrm{des}_1}
\def\ed{\mathrm{des}_0}
\def\co{\mathrm{co}}
\numberwithin{equation}{section}
\date{\today}
\usepackage{color}
\usepackage[dvipsnames]{xcolor}

\usepackage{tikz}
\def\And{\mathrm{And}}
\def\and{\mathrm{d}}
\begin{document}

\title[Enumeration of permutations  by the parity of descent positions]{Enumeration of permutations  by the parity of descent positions}
\author{Qiongqiong Pan and Jiang Zeng}
\address{College of Mathematics and Physics, Wenzhou University\\
Wenzhou 325035, PR China}
\email{qpan@wzu.edu.cn}
\address{Universit\'e de Lyon,  Universit\'e Lyon 1,
UMR 5208 du CNRS, Institut Camille Jordan\\
F-69622, Villeurbanne Cedex, France}
\email{zeng@math.univ-lyon1.fr}

\begin{abstract} 
Noticing  that  some recent variations of descent polynomials 
are special cases of   Carlitz and Scoville's 
four-variable polynomials, which   enumerate
permutations by  the parity of descent and ascent positions,
we prove a $q$-analogue 
  of Carlitz-Scoville's generating function  by counting the inversion number and a type B analogue by enumerating the signed permutations with respect to the parity of desecnt and ascent positions. 
As a by-product of our formulas, we obtain a $q$-analogue of Chebikin's formula for alternating descent polynomials, an alternative proof of 
Sun's gamma-positivity   of  her bivariate Eulerian polynomials
and  a type B analogue, the latter refines  Petersen's gamma-positivity of the type B Eulerian polynomials. 
\end{abstract}

\maketitle

\tableofcontents

\section{Introduction}

In the past few years of this century, several variations and refinements of   
permutation descent,   according to  the parity of descent positions, have been studied,  
 see \cite{Che08, Re12, GZ14, MY16, Sun18, SZ19, Sun21, MFMY22, LMWW22, Pan22}. 
 This paper arose from the observation that some of these  results  are related to
  a work of Carlitz and Scoville~\cite{CS73} dated back to 1973. For example, Chebikin's \emph{alternating descent polynomial} \cite{Che08} 
 and the  \emph{bivariate 
  Eulerian polynomials}   in  H. Sun~\cite{Sun18} and Y. Sun and Zhai~\cite{SZ19} are both 
  special cases of Carlitz-Scoville's four-variable polynomials enumerating the permutations according to the parity of both descents and ascents.  
  On the other hand, this  connection  leads immediately to obtain two equivalent simpler versions of Carlitz-Scoville's generating function.
  As Carlitz and Scoville's original proof  
 relies on solving a system of differential equations, 
this prompted us to  find  a more conceptuel proof, which led up straightforwardly  to a $q$-analogue.

If $\pi$  is a permutation of $[n]:=\{1, \ldots, n\}$, 
an index $i\in [n-1]$ is a \emph{descent position} (resp. \emph{ascent position}) of $\pi$ if $\pi(i)>\pi(i+1)$ (resp. $\pi(i)<\pi(i+1)$).  Let $\des \pi$ (resp. $\odes\pi$ and $\edes\pi$)
 be the number of descents of $\pi$ (resp.  at odd and even  positions), i.e.,
 $$
\des_{\nu}(\pi)=\#\{i\in [n]| \pi(i)>\pi(i+1) \; {\rm and}\; i\equiv \nu \pmod{2}\}\quad (\nu\in \{0,1\}).
 $$
  The statistics  $\rise \pi$, $\orise\pi$ and $\erise\pi$ are defined similarly. For $i\in\{2,3,\ldots, n-1\}$, we say $\pi(i)$ is a valley (resp. peak) of $\pi$, if $\pi(i-1)>\pi(i)<\pi(i+1)$ (resp. $\pi(i-1)<\pi(i)>\pi(i+1)$) and $\pi(i)$ is a double ascent (resp. double descent) of $\pi$, if $\pi(i-1)<\pi(i)<\pi(i+1)$ (resp. $\pi(i-1)>\pi(i)>\pi(i+1)$). Finally 
we recall that the inversion number of $\pi$ is   
$\inv\pi =|\{(i,j)| \pi(i)>\pi(j), \; 1\leq i<j\leq n\}|$.
  
Define the  enumerative  polynomial of permutations of $\S_n$  by the parity of ascent and descent positions as 
$$
P_n(x_0, x_1, y_0, y_1,q)=\sum_{\sigma\in \S_n} x_0^{\erise\sigma}x_1^{\orise\sigma}y_0^{\edes\sigma}y_1^{\odes\sigma}q^{\inv\sigma}.
$$
Recall the following $q$-exponential series
$$
\exp_q(x)=\sum_{n\geq 0} \frac{x^n}{n!_q},
$$
where $0!_q=1$ and  $n!_q=\prod_{i=1}^n(1+q+\cdots +q^{i-1})$ for $n\geq 1$, and the $q$-trignometric series
\begin{align*}
\cosh_q t&=\sum_{n\geq 0}\frac{t^{2n}}{(2n)!_q}, \quad 
\sinh_q t=\sum_{n\geq 1}\frac{t^{2n-1}}{(2n-1)!_q};\\
\cos_q x&=\sum_{n=0}^\infty (-1)^n\frac{x^{2n}}{(2n)!_q},\quad
\sin_q x=\sum_{n=1}^\infty (-1)^{n-1}\frac{x^{2n-1}}{(2n-1)!_q}.
\end{align*}
\begin{thm}\label{thmA3} Let 
$\alpha=\sqrt{(y_0-x_0)(y_1-x_1)}$.
Then
\begin{gather}
\label{q-CS}
\sum_{n\geq 1}P_{n}(x_0, x_1, y_0, y_1,q)\frac{t^{n}}{n!_q}\nonumber\\
=\frac{(x_1+y_1)\cosh_q(\alpha t)+\alpha \sinh_q(\alpha t)-y_1 (\cosh_q^2(\alpha t)-\sinh_q^2(\alpha t))-x_1}
{x_0x_1-(x_0y_1+x_1y_0)\cosh_q(\alpha t)+y_0y_1 (\cosh_q^2(\alpha t)-\sinh_q^2(\alpha t))}.
\end{gather}
\end{thm}
\begin{remark}
When $q=1$ Eq.~\eqref{q-CS} reduces to  Carlitz-Scoville's formula~\cite[Theorem~3.1]{CS73}~\footnote{Carlitz and Scoville counted a conventional rise at the beginning as   position $0$ and a conventional descent at the end as  position $n\pmod{2}$.}
\begin{align}\label{CS}
\sum_{n\geq 1}P_{n}(x_0, x_1, y_0, y_1,1)
\frac{t^{n}}{n!}=
\frac{(x_1+y_1)\sum_{n\geq 1}\frac{\beta^{n-1}t^{2n}}{(2n)!}+\sum_{n\geq 1}\frac{\beta^{n-1}\,t^{2n-1}}{(2n-1)!}}
{1-(x_0y_1+x_1y_0)\sum_{n\geq 1} \frac{\beta^{n-1}t^{2n}} {(2n)!}},
\end{align}
with $\beta=(y_0-x_0)(y_1-x_1)$. For the homegeous Eulerian polynomials
 $P_n(y,y,x,x,1)$, i.e., $\sum_{\sigma\in \S_n} 
x^{\des \sigma} y^{\rise \sigma}$,   the corresponding formula reads
\begin{align}\label{cf}
\sum_{n\geq 1}P_n(y,y,x,x,1)\frac{t^n}{n!}=\frac{e^{xt}-e^{yt}}{xe^{yt}-ye^{xt}}.
\end{align}
Chen and Fu~\cite{CF22} recently 
gave a  context-free  grammar proof of  \eqref{cf}.
\end{remark}
Let $\rm{UD}_n$ be  the set of \emph{up-down} permutations of
$12\ldots n$, 
 i.e.,  permutations $\sigma:=\sigma(1)\ldots \sigma(n)$ such that 
$\sigma(1)<\sigma(2)>\sigma(3)<\cdots$. 
Obviously  
$$P_{n}(0,1, 1,0,q)
=\sum_{\sigma \in \rm{UD}_n} q^{\inv \sigma}$$
and  Eq.~\eqref{q-CS} reduces to a  $q$-analogue of 
Andr\'e's classical result (see \cite{EC1, GJ04,  JV15})~:
\begin{align}\label{andre}
1+\sum_{n\geq 1} P_{n}(0,1, 1,0,q) \frac{x^n}{n!_q}=\frac{1+\sin_q x}{\cos_q x}.
\end{align}
For  the following two special cases:
\begin{subequations}\label{sun-chebikin-qq}
\begin{align}\label{sunhua-polynomial-q}
A_n(x,y,q):=P_n(1,1, y,x,q)=\sum_{\sigma\in \S_n} x^{\odes\sigma}y^{\edes\sigma}q^{\inv \sigma},\\
\widehat{A}_n(x,y,q):=P_n(y,1, 1,x,q)=\sum_{\sigma\in \S_n} x^{\odes\sigma}y^{\erise\sigma}q^{\inv \sigma},\label{chebikin-polynomial-q}
\end{align} 
\end{subequations}
we derive from Theorem~\ref{thmA3} that
\begin{subequations}\label{gf-sun-chebikin-q}
\begin{align}\label{pz1}
\sum_{n\geq 1}A_{n}(x,y,q)\frac{t^{n}}{n!_q}&=
\frac{(1+x)\cosh_q(\alpha t)+\alpha \sinh_q(\alpha t)-x (\cosh_q^2(\alpha t)-\sinh_q^2(\alpha t))-1}
{1-(x+y)\cosh_q(\alpha t)+xy (\cosh_q^2(\alpha t)-\sinh_q^2(\alpha t))},\\
\sum_{n\geq 1}\widehat{A}_{n}(x,y,q)\frac{t^{n}}{n!_q}&
=
\frac{(1+x)\cos_q(\alpha t)-\alpha \sin_q(\alpha t)-x (\cos_q^2(\alpha t)+\sin_q^2(\alpha t))-1}
{y-(xy+1)\cos_q(\alpha t)+x(\cos_q^2(\alpha t)+\sin_q^2(\alpha t))}\label{pz2}
\end{align}
\end{subequations}
with $\alpha=\sqrt{(1-x)(1-y)}$.

\begin{remark}
Formulae~\eqref{CS}, \eqref{pz1} and ~\eqref{pz2} are actually  equivalent.
Indeed, 
for any  $\sigma\in \S_n$ it is clear that
\begin{subequations}
\begin{align}
\edes\sigma+\erise\sigma&=\lfloor (n-1)/2\rfloor,\label{R1}\\
\odes\sigma+\orise\sigma&=\lfloor n/2\rfloor. \label{R2}
\end{align}
\end{subequations}
Hence the distribution of the 
 quadruple statistics $(\erise, \orise,\edes,\odes)$ is equivalent to any  pair of the statistics  in 
$
\{\odes,\orise\}\times  \{\edes,\erise\}$. In particular, we have
\begin{align}\label{tildeA:A}
\widehat{A}_{n}(x,y,q)&=y^{\lfloor (n-1)/2\rfloor}A_n(x, 1/y,q),\\
\intertext{and}
P_{n}(x_0, x_1, y_0, y_1,q)
&=x_0^{\lfloor (n-1)/2\rfloor}x_1^{\lfloor n/2\rfloor} A_{n}\left(\frac{y_1}{x_1}, \frac{y_0}{x_0},q\right).
\end{align}
The polynomial  
 $A_n(x,x,q):=\sum_{\sigma\in \S_n} x^{\des\sigma}q^{\inv\sigma}$ is a classical 
  $q$-analogue of Eulerian polynomials and Eq.~\eqref{pz1} yields  Stanley's formula~\cite{St76, Pe15},
 \begin{align}
 1+\sum_{n\geq 1}xA_n(x,x,q) \frac{t^n}{n!_q}=\frac{1-x}{1-x\exp_q((1-x)t)},
\end{align}
of which another  refinement was  given in \cite{PZ19}.

As a variation  of descent,  Chebikin~\cite{Che08} introduced
 the \emph{alternating descent set} of  permutation $\pi\in \S_n$ by
 $$
 \widehat{D}(\pi)=\{i\in [n-1]| \pi(i)>\pi(i+1)\; \textrm{and $i$ is odd or}\;
\pi(i)<\pi(i+1)\;   \textrm{and $i$ is even}\}.
$$
Hence, 
  the number of alternating descents 
  $\widehat{\des}\pi=|\widehat{D}(\pi)|$ equals 
 $\des_1\sigma+\rise_0\sigma$ and 
formula \eqref{pz2} with $x=y$ and $q=1$  reduces to  
\begin{align}\label{eq:che}
1+\sum_{n\geq 1}x\widehat{A}_{n}(x,x,1)\frac{t^{n}}{n!}=
\frac{1-x}{1-x(\sec (1-x)t+\tan (1-x)t)},
\end{align}
which   is equivalent to \cite[Theorem 4.2 ]{Che08}, see also \cite[Eq. (22)]{GZ14}.
As Chebikin, being unaware of  the work of Carlitz and Scoville, 
 Sun~\cite{Sun18} and 
Sun and Zhai~\cite{SZ19}  reconsidered  the polynomials $A_n(x,y,1)$, and a cumbersome formula for \eqref{pz1} is given in 
\cite[Theorem 2.2]{SZ19}. Other proofs of formula \eqref{eq:che}  and generalizations appeared  in \cite{Re12,GZ14, LMWW22, Pan22}. 
\end{remark}

As the original proof of \eqref{CS} with $q=1$ in  \cite{CS73} is not easy (see also the solution of Exercise 4.3.14 in \cite{GJ04}),
we shall give a more conceptual   proof of \eqref{pz1}, which is equivalent to Theorem~\ref{CS},  by exploring   a \emph{sieve method}, see \cite{St76, Ga79, Che08, EC1}.

 Our second goal is to give a type B analogue of 
 Carlitz and Scoville's formula, i.e., 
 Theorem~\ref{thmA3} with $q=1$.
Denote by $\mathcal{B}_n$ the collection of type B permutations $\sigma$  of the set $[\pm n]:=\{\pm 1,\ldots,\pm n\}$ such that $\sigma(-i)=-\sigma(i)$ for all $i\in [n]$, obviously, $|\sigma|:=|\sigma(1)|\ldots|\sigma(n)|\in\S_n$. As usual (see \cite{BB05, Pe15}), we always assume that type B permutations are prepended by $0$. That is, we identify an element $\sigma=\sigma(1)\ldots\sigma(n)$ in $\mathcal{B}_n$ with the word $\sigma(0)\sigma(1)\ldots\sigma(n)$, where $\sigma(0)=0$. 
We say that $\sigma\in\mathcal{B}_n$ has a descent (resp. ascent) at position $i$, if  $\sigma(i)>\sigma(i+1)$  (resp.
$\sigma(i)<\sigma(i+1)$) for $i\in\{0\}\cup[n-1]$. By abuse of notation, in this section,  we use 
$\des \sigma$ (resp. $\odes\sigma$ and $\edes\sigma$) to denote 
 the number of descents of $\sigma$ (resp.  at odd and even  positions). The statistics  $\rise \sigma$,
 $\orise\sigma$ 
 and $\erise\sigma$ are defined similarly for the ascents. 

Define  the enumerative polynomials
\begin{align}\label{sunhua-polynomial-B}
B_n(x,y):=\sum_{\sigma\in \B_n} x^{\odes\sigma}y^{\edes\sigma}.
\end{align} 
\begin{thm}\label{thmB1} Let $\alpha=\sqrt{(1-x)(1-y)}$. Then
\begin{subequations}
\begin{align}\label{eq:B1}
&\sum_{n\geq1}B_{2n}(x,y)\frac{t^{2n}}{(2n)!}
=\frac{(x+y)\cosh(2\alpha t)+(1-x)(1-y)\cosh(\alpha t)-(1+xy)}{(1+xy)
-(x+y)\cosh(2\alpha t)},\\
&\sum_{n\geq 1}B_{2n-1}(x,y)\frac{t^{2n-1}}{(2n-1)!}
=\frac{\alpha(1+y)\,\sinh(\alpha t)}{(1+xy)-(x+y)\cosh(\alpha t)}.\label{eq:B2}
\end{align}
\end{subequations}
\end{thm}

\begin{remark}
When $x=y$, the polynomial
 $B_n(x,x):=\sum_{\sigma\in \B_n} x^{\des\sigma}$ is the 
  usual \emph{Eulerian polynomial of type B} and Theorem~\ref{thmB1} is equivalent to the known generating 
 function, see \cite[Corollary 3.9]{CG07} or 
 \cite[Theorem 13.3]{Pe15},
 \begin{align}
 \sum_{n\geq 0}B_n(x,x) \frac{t^n}{n!}=\frac{(x-1)e^{t(x-1)}}{x-e^{2t(x-1)}}.
\end{align}
\end{remark}

Now, consider the following variant of $B_n(x,y)$
\begin{equation}
\widehat{B}_n(x,y):=\sum_{\sigma\in \B_n} x^{\odes\sigma}y^{\erise\sigma}=y^{\lfloor (n+1)/2\rfloor}B_n(x, 1/y).
\label{chebikin-polynomial-B}
\end{equation}
From Theorem~\ref{thmB1} we derive plainly the  generating function of the latter polynomials. 
\begin{thm}\label{thmB2}Let $\alpha=\sqrt{(1-x)(1-y)}$.  Then
\begin{subequations}
\begin{align}\label{z1}
&\sum_{n\geq 1}\widehat{B}_{2n}(x,y)\frac{t^{2n}}{(2n)!}
=
\frac{(1+xy)\cos(2\alpha t)-(1-x)(1-y)\cos(\alpha t)-(x+y)}{(x+y)-(1+xy)\cos(2\alpha t)},\\
\label{z2}
&\sum_{n\geq 1}\widehat{B}_{2n-1}(x,y)\frac{t^{2n-1}}{(2n-1)!}
=\frac{-\alpha(1+y)\,\sinh(\alpha t)}{(x+y)-(1+xy)\cos(2\alpha t)}.
\end{align}
\end{subequations}

\end{thm}

\begin{remark}

Similar to Chebikin's   alternating descent set of type A (see \cite{Che08}), 
we can define the 
  \emph{alternating descent set } of any $\sigma\in\mathcal{B}_n$  by
$$
 \widehat{D}_B(\pi)=\{i\in \{0\}\cup [n-1]|
  \pi(i)>\pi(i+1)\; \textrm{if $i$ is odd or}\;
\pi(i)<\pi(i+1)\;   \textrm{if $i$ is even}\}.
$$  
Let $\widehat{\des}_B(\sigma)=|\widehat{D}_B(\sigma)|$.  Clearly 
$\widehat{B}_n(x,x)=\sum_{\sigma\in\mathcal{B}_n}
 x^{\widehat{\des}_{B}(\sigma)}$, which is 
the  $n$-th \emph{alternating Eulerian polynomial of type B} 
in \cite{MY16}, and
Theorem~\ref{thmB2} reduces to  the generating function in  \cite{MFMY22,  DZ21,Pan22},
\begin{align}\label{p6}
\sum_{n\geq0}\hat{B}_n(x,x)\frac{u^n}{n!}=\frac{x-1}{(x-1)\cos(u(1-x))+(x+1)\sin(u(1-x))}.
\end{align}
\end{remark}
Define  the  general enumerative  polynomials of permutations by the parity of the ascent and descent positions:
\begin{equation}\label{QB}
P^B_n(x_0, x_1, y_0, y_1)=\sum_{\sigma\in \mathcal{B}_n} x_0^{\erise\sigma}x_1^{\orise\sigma}y_0^{\edes\sigma}y_1^{\odes\sigma}.
\end{equation}
For any  $\sigma\in \B_n$ we have 
\begin{equation}\label{linkB}
\begin{split}
\edes\sigma+\erise\sigma=\lfloor (n+1)/2\rfloor,\\
\odes\sigma+\orise\sigma=\lfloor n/2\rfloor.
\end{split}
\end{equation}
Hence the distribution of the 
 quadruple statistics $(\erise, \orise,\edes,\odes)$ is equivalent to any of the four pairs in 
$
\{\odes,\orise\}\times  \{\edes,\erise\}$.
It follows that 
\begin{align}\label{PBB}
P^B_{n}(x_0, x_1, y_0, y_1)
&=x_0^{\lfloor (n+1)/2\rfloor}x_1^{\lfloor n/2\rfloor} B_{n}\left(\frac{y_1}{x_1}, \frac{y_0}{x_0}\right).
\end{align}
We derive plainly the following generating function from Theorem~\ref{thmB1}.
\begin{thm}\label{thmB3} We have 
\begin{align}
\sum_{n\geq 1}P^B_{2n}(x_0, x_1, y_0, y_1)\frac{t^{2n}}{(2n)!}=\frac{(x_0y_1+x_1y_0)\sum_{n\geq 0}\frac{\alpha^n (2t)^{2n}}{(2n)!}+\sum_{n\geq 0}\frac{\alpha^{n+1} t^{2n}}{(2n)!}-(x_1x_0+y_0y_1)}
{(x_0x_1+y_0y_1)-(y_1x_0+x_1y_0)\sum_{n\geq 0}\frac{\alpha^n (2t)^{2n}}{(2n)!}},
\end{align}
and 
\begin{align}
\sum_{n\geq 1}P^B_{2n-1}(x_0, x_1, y_0, y_1)\frac{t^{2n-1}}{(2n-1)!}=
\frac{(y_0^2-x_0^2)(y_1-x_1)\,\sum_{n\geq 0}\frac{\alpha^{n} t^{2n+1}}{(2n+1)!}}
{(x_0x_1+y_0y_1)-(x_0y_1+x_1y_0)\sum_{n\geq 0}\frac{\alpha^n (2t)^{2n}}{(2n)!}},
\end{align}
where $\alpha=(y_0-x_0)(y_1-x_1)$.
\end{thm}
In view of \eqref{chebikin-polynomial-B} and \eqref{PBB}, 
Theorem~\ref{thmB1}, Theorem~\ref{thmB2} and Theorem~\ref{thmB3} are equivalent. We shall give a proof of  Theorem~\ref{thmB1} in the same  vein as the  proof of \eqref{pz1} with $q=1$.

An important feature of Eulerian polynomials is the gamma-nonnegativity~\cite{Pe15}.   More recently, Sun~\cite{Sun21} proved that 
the bivariate Eulerian polynomials $(1+y)A_{2n}(x,y,1)$ and $A_{2n+1}(x,y,1)$  are  $\gamma$-positive (see Theorem~\ref{thm:sun}).  Our third goal is to  
derive some symmetric expansion formulae for 
 bivariate polynomials allied to  the above four families of bi-Eulerian 
 polynomials. This will be done by applying   
their  generating functions and combinatorics of Andr\'e permutations~\cite{FS71, FS76, HR98}.

The rest of this paper is organised as follows. 
We will first 
 study the symmetric and gamma
expansions of  the two   sequences of bi-polynomials as well as their type analogues
 in Section~2 and postpone the proof of   \eqref{pz1} and Theorem~\ref{thmB1} to Section~3 and Section~4, respectively.  
We conclude with   some open  problems in Section~5.

As suggested by a referee, 
for reader's convenience, we  list the main permutation statistics of this paper  in the following table.
\begin{table}[h!]
\begin{tabular}{|c|c|}
\hline
$\edes\pi$& the number of descents of $\pi$ at even positions \\
\hline
$\odes\pi$& the number of descents of $\pi$ at odd positions \\
\hline
$\erise\pi$& the number of ascents of $\pi$ at even positions \\
\hline
$\orise\pi$& the number of ascents of $\pi$ at odd positions \\
\hline
$\inv\pi$& the number of inversions  of $\pi$ \\
\hline
$\lpk(\pi)$& the number of left peaks of $\pi$, see \eqref{lpk} \\
\hline
\end{tabular}
\medskip
\caption{Main statistics of $\pi\in \S_n$ }
\end{table}
\section{Symmetric and positive expansions of bi-Eulerian polynomials}
Define  two families of bi-Eulerian polynomials $(\widetilde{A}_n(x,y))_{n\geq 1}$
and $(\overline{A}_n(x,y))_{n\geq 1}$
by
\begin{subequations}\label{ST}
\begin{align}\label{S-A}
\widetilde{A}_{2n}(x,y)&=(1+y) A_{2n}(x,y,1), \quad 
\widetilde{A}_{2n-1}(x,y)=A_{2n-1}(x,y,1),\\
\overline{A}_{2n}(x,y)&=(1+y) \widehat{A}_{2n}(x,y,1), \quad 
\overline{A}_{2n-1}(x,y)=\widehat{A}_{2n-1}(x,y,1); \label{T-hatA}
\end{align}
\end{subequations}
and their type B analogues
$(\widetilde{B}_n(x,y))_{n\geq 1}$
and $(\overline{B}_n(x,y))_{n\geq 1}$
by
\begin{subequations}
\begin{align}\label{B1}
\widetilde{B}_{2n}(x,y)&=B_{2n}(x,y), \quad 
\widetilde{B}_{2n-1}(x,y)=(1+y)^{-1}B_{2n-1}(x,y),\\\label{B2}
\overline{B}_{2n}(x,y)&= \widehat{B}_{2n}(x,y), \quad 
\overline{B}_{2n-1}(x,y)=(1+y)^{-1}\widehat{B}_{2n-1}(x,y).
\end{align}
\end{subequations}
By  \eqref{pz1} and \eqref{pz2} (resp. Theorem~\ref{thmB1} and Theorem~\ref{thmB2})
both  polynomials  $\widetilde{A}_{n}(x,y)$
and $\overline{A}_{n}(x,y)$ (resp.  $\widetilde{B}_{n}(x,y)$
and $\overline{B}_{n}(x,y)$ ) are symmetric in $x$ and $y$.

Recall that a polynomial with real coefficients  
$P(x)=\sum_{i=0}^n a_i x^i$ is 
\emph{gamma-positive} (resp. \emph{semi-gamma-positive}) if there are nonnegative numbers 
$\gamma_{i}$ such that 
$P(x)=\sum_{i}\gamma_{i} x^i(1+x)^{n-2i}$ 
(resp. $P(x)=(1+x)^\nu\sum_{i}\gamma_{i} x^i(1+x^2)^{\lfloor n/2\rfloor-i}$ with $\nu=0$ or 1.), see \cite{Pe15} and \cite{MMY20} respectively. It is known that the gamma-positivity is stronger than the semi-gamma-positivity~\cite{MMY20}.

In this section,  we shall first derive the semi-gamma-positive formulae for the bi-Eulerian polynomials 
$\widetilde{A}_n(x,y), \overline{A}_n(x,y), \widetilde{B}_n(x,y)$
and  $\overline{B}_n(x,y)$ from their generating functions  and then  apply  Hetyei-Reiner's min-max tree model~\cite{HR98} for permutations  to derive 
the corresponding $\gamma$-positive formulae for  $\widetilde{A}_n(x,y)$ and $\overline{A}_n(x,y)$ as well as their type B analogues 
 by refining Petersen's proof for the $\gamma$-positivity of type B Eulerian polynomials~\cite{Pe15}.

\subsection{Semi-gamma-positivity of bi-Eulerian polynomials}

%
The following  generalizes the semi-gamma-positivity of Eulerian polynomials to bi-Eulerian polynomials.
\begin{thm}\label{thm-typeA}
Let  $a(n,j)$ (resp. $\bar a(n,j)$) be  the number of permutations in $\S_n$ with $j$ odd  descents and without even descents (resp. ascents)
for $n\geq 1$  and $0\leq 2j\leq n$.  Then
\begin{subequations}
\begin{align}\label{sym-S}
\widetilde{A}_n(x,y)&=\sum_{j=0}^{\lfloor \frac{n}{2}\rfloor}
a(n,j)\,(x+y)^j (1+xy)^{\lfloor \frac{n}{2}\rfloor-j};\\
\overline{A}_n(x,y)&=\sum_{j=0}^{\lfloor \frac{n}{2}\rfloor}
\bar a(n,j)\,(x+y)^j (1+xy)^{\lfloor \frac{n}{2}\rfloor-j},\label{sym-T}
\end{align}
and
\begin{align}\label{a-a-bar}
\bar a(n,j)=a(n, \left\lfloor n/2\right\rfloor-j)\quad  \textrm{for}\quad  0\leq j\leq \left\lfloor {n/2}\right\rfloor.
\end{align}
\end{subequations}
\end{thm}
\begin{proof}
Let $\alpha(x,y)=(1-x)(1-y)$. Then 
$$
\alpha(x,y)=(1+xy)\cdot  \alpha\biggl(\frac{x+y}{1+xy},0\biggr).
$$
It follows   from \eqref{gf-sun-chebikin-q} that 
\begin{subequations}
\begin{align}
\widetilde{A}_n(x,y)&=(1+xy)^{\lfloor \frac{n}{2}\rfloor}
 A_{n}\left(\frac{x+y}{1+xy},0, 1\right),\\
\overline{A}_n(x,y)&=(1+xy)^{\lfloor \frac{n}{2}\rfloor}
 \widehat{A}_{n}\left(\frac{x+y}{1+xy},0, 1\right),
\end{align}
\end{subequations}
which are obviously equivalent to   \eqref{sym-S}  and \eqref{sym-T}, respectively.

Define the completion $\sigma^c$ of $\sigma\in \S_n$  
by $\sigma^c(i)=n+1-\sigma(i)$ for $1\leq i\leq n$. It is clear that 
the mapping  $\varphi: \sigma\mapsto \sigma^c$  is an involution on $\S_n$ and satisfies 
$
\ides\sigma=\iasc \sigma^c
$ for $i\in \{0, 1\}$. Thus  
$$
(\des_1\sigma^c,\; \rise_0\sigma^c)=(\rise_1\sigma, \;\des_0\sigma)
=\left(\lfloor{n/2}\rfloor-\des_1\sigma,\; \des_0\sigma\right).
$$
Eq. \eqref{a-a-bar} follows  by restricting $\varphi$ on the set of permutations in $\S_n$ with $j$ odd descents and without even descent.
\end{proof}
\begin{remark} 
The combinatorial interpretation of $a_{n,j}$ actually follows from 
the existence of  formula~\eqref{sym-S}, which  
 was first  conjectured by  Sun~\cite{Sun18} and then  proved by 
 Sun and Zhai~\cite{SZ19}.  
\end{remark}

Similarly, we have the following B-analogue of Theorem~\ref{thm-typeA}.

\begin{thm}\label{TB1}
Let $b(n,j)$ (resp. $\bar b(n,j)$) be  the number of permutations in $\B_n$ with $j$ odd  descents and without even descents (resp. even ascents). 
 Then 
 \begin{subequations}
\begin{align}\label{sym-U}
\widetilde{B}_{n}(x,y)&=\sum_{j=0}^{\lfloor \frac{n}{2}\rfloor}
b(n,j)\,(x+y)^j (1+xy)^{\lfloor \frac{n}{2}\rfloor-j},\\
\overline{B}_{n}(x,y)&=\sum_{j=0}^{\lfloor \frac{n}{2}\rfloor}
\bar b(n,j)\,(x+y)^j (1+xy)^{\lfloor \frac{n}{2}\rfloor-j},\label{sym-V}
\end{align}
and
\begin{align}\label{b-b-bar}
\bar b(n,j)=b(n, \left\lfloor n/2\right\rfloor-j)\quad  \textrm{for}\quad  0\leq j\leq \left\lfloor {n/2}\right\rfloor.
\end{align}
\end{subequations}
\end{thm}
\begin{proof}
Let $\alpha(x,y)=(1-x)(1-y)$. Then 
$$
\alpha(x,y)=(1+xy)\cdot  \alpha((x+y)/(1+xy),0).
$$
We derive  from Theorem~\ref{thmB1} and Theorem~\ref{thmB2} immediately
\begin{subequations}
\begin{align}
\widetilde{B}_{n}(x,y)=(1+xy)^{\lfloor \frac{n}{2}\rfloor}
 B_{n}\left(\frac{x+y}{1+xy},0\right),\\
 \overline{B}_{n}(x,y)=(1+xy)^{\lfloor \frac{n}{2}\rfloor}
\widehat{B}_{n}\left(\frac{x+y}{1+xy},0\right),
\end{align}
\end{subequations}
which are what  \eqref{sym-U} and \eqref{sym-V} mean.

Consider  the negation $\bar \sigma$ of $\sigma\in \B_n$  
by $\bar \sigma(i)=-\sigma(i)$ for $1\leq i\leq n$. It is clear that 
the mapping  $\phi: \sigma\mapsto \bar \sigma$  is an involution on $\S_n$ and satisfies 
$\ides\sigma=\iasc \bar \sigma$ for $i\in \{0, 1\}$. Thus
$$
(\des_1\bar \sigma,\; \rise_0\bar \sigma)=(\rise_1\sigma, \;\des_0\sigma)
=\left(\lfloor{n/2}\rfloor-\des_1\sigma,\; \des_0\sigma\right).
$$
Eq. \eqref{b-b-bar} follows  by restricting $\phi$ on the set of permutations in $\B_n$ with $j$ odd descents and without even descent.
\end{proof}


We note that Eq.~\eqref{sym-S} does not   directly reduce
 to the known $\gamma$-positivity formula of Eulerian polynomials $A_n(x,x)$ when $x=y$. To derive the latter expansion 
we shall appeal to the min-max tree representations of permutations due to 
Hetyei and Reiner~\cite{HR98}.
Similarly, to derive  gamma-positivity formula  of 
 type B Eulerian polynomials $B_n(x,x)$   from Theorem~\ref{TB1} we shall appeal 
   to an action on permutations due to  Petersen~\cite{Pe15}.


\subsection{Gamma-positivity  of 
bi-Eulerian polynomials of type A}
We define the \emph{min-max tree} $M(w)$ 
associated to a sequence of 
distincte integers $w=w_1\ldots w_n$ as follows. 
\begin{itemize}
\item [(1)]
First, $M(w)$ is a binary tree with vertices labelled $w_1,\ldots, w_n$. Let $i$ be the least integer for which either $w_i=min\{w_1,w_2,\ldots,w_n\}$ or $w_i=\max\{w_1,w_2,\ldots,w_n\}$. Define $w_i$ to be the root of $M(w)$.
\item[(2)]
Then recursively define $M(w_1,\ldots,w_{i-1})$ and $M(w_{i+1},\ldots,w_n)$ 
to be the left and right subtree of $w_i$, respectively.
\end{itemize}
 Conversely, the left-first order reading of the tree $M(w)$ yields the sequence $w$, see \cite{HR98, FH16} and \cite[pp. 57-61]{EC1}.

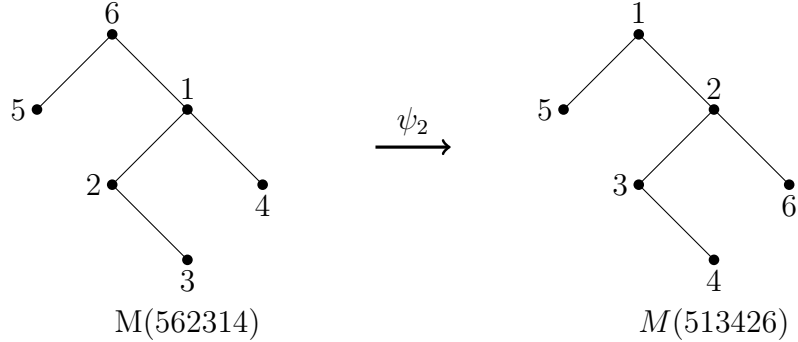
\begin{figure}
\begin{tikzpicture}
\draw (0,0)--(-1,-1);
\draw (0,0)--(1,-1)--(2,-2);
\draw (1,-1)--(0,-2)--(1,-3);
\fill (0,0) circle (2pt);
\node[above] at (0,0) {$6$};
\node[left] at (-1,-1) {$5$};
\node[above] at (1,-1) {$1$};
\node[left] at (0,-2) {$2$};
\node[below] at (1,-3) {$3$};
\node[below] at (2,-2) {$4$};
\fill (-1,-1) circle (2pt);
\fill (1,-1) circle (2pt);
\fill (2,-2) circle (2pt);
\fill (0,-2) circle (2pt);
\fill (1,-3) circle (2pt);
\node[below] at (1,-3.5) {M(562314)};
\draw[][->,very thick](3.5,-1.5)--(4.5,-1.5) node[midway, above]{$\psi_2$};
\draw (7,0)--(6,-1);
\draw (7,0)--(8,-1)--(9,-2);
\draw (8,-1)--(7,-2)--(8,-3);
\fill (7,0) circle (2pt);
\node[above] at (7,0) {$1$};
\node[left] at (6,-1) {$5$};
\node[above] at (8,-1) {$2$};
\node[left] at (7,-2) {$3$};
\node[below] at (8,-3) {$4$};
\node[below] at (9,-2) {$6$};
\fill (6,-1) circle (2pt);
\fill (8,-1) circle (2pt);
\fill (9,-2) circle (2pt);
\fill (7,-2) circle (2pt);
\fill (8,-3) circle (2pt);
\node[below] at (8,-3.5) {$M(513426)$};
\end{tikzpicture}
\caption{ The action of operator  $\psi_2$ at tree $M(562314)$.}\label{MNT}
\end{figure}


An interior vertex in $M(w)$ is called a \emph{min~(resp. max)} vertex if it is the minimum~(resp. maximum) label among all its descendants.
Let $M(w_i)~(resp.~M_l(w_i),M_r(w_i))$ denote the subtree~(resp. the left subtree, the right subtree) of  $M(w)$ with root $w_i$.

For $1\leq i\leq n$, we define the
 \emph{operator} 
$\psi_i$  permuting the labels of $M(w)$ as in the following.
\begin{enumerate}
\item 
If $w_i$ is a min vertex, then replace $w_i$ by the largest element of $M_r(w_i)$, permute the remaining elements of $M_r(w_i)$ such that they keep their same relative orders and all other vertices in $M(w)$ are fixed.
\item
If $w_i$ is a max vertex, then replace $w_i$ by the smallest element of $M_r(w_i)$ such that they keep their same relative order, and all other vertices in $M(w)$ are fixed.
\end{enumerate}
 An illustration of operator $\psi_2$  is  given  in Figure~\ref{MNT}.

Given a permutation  $\pi=\pi(1)\pi(2)\ldots\pi(n)$  
 of   $Y=\{y_1, y_2, \ldots,y_n\}_<$, which is a set of positive integers.
The $\pi(i)$-factorization of $\pi$ is  the sequence $(w_1,w_2,\pi(i),w_4,w_5)$,  $1\leq i\leq n$, where
\begin{itemize}
\item[(1)]
the concatenation product $w_1w_2\pi(i) w_4w_5$ is equal to $\pi$;
\item[(2)] $w_2$ is the longest right factor of $\pi(1)\pi(2)\ldots \pi(i-1)$, all letters of which are greater than $\pi(i)$;
\item[(3)] $w_4$ is the longest left factor of $\pi(i+1)\pi(i+2)\ldots \pi(n)$, all letters of which are greater than $\pi(i)$. 
\end{itemize}
Note that above any of $w_1$, $w_2$, $w_4$ or $w_5$ may be empty.

\begin{definition}[see \cite{FS76, FH16}]
A permutation $\pi\in\S_n$ is an \emph{Andr\'e  permutation} (of  kind I) if 
$\pi$ has no double descents and  ends with ascent, i.e., $\pi(n-1)<\pi(n)$, and 
if  $i\in\{2,\ldots,n\}$ is a valley of $\pi$ and $(w_1,w_2,\pi(i), w_4,w_5)$ is the 
 $\pi(i)$-factorization of $\pi$, then  the maximum letter of $w_2w_4$ is in  $w_4$.
\end{definition}
For example,  the Andr\'e permutations of length $4$ are 
$1234,\,1324,\,2314,\,2134$ and $3124$.

\begin{fact}
The  operators $\psi_i$  are commuting  
involutions acting on $M(w)$  and generate an abelien group $G_w$ isomorphic to  $(\Z/2\Z)^{l(w)}$, where $l(w)$ is the number of internal certices of $M(w)$. Those $\psi_i$ for which $w_i$ is an internal vertex are a minimal set $S_w$ of generators for $G_w$.
For any subset $S\subseteq G_w$ we define 
the HR action $\psi_S$
by 
$\psi_S(M(w))=\prod_{i\in S}\psi_i(M(w))$.
\end{fact}


For $\pi\in\S_n$, let $\mathrm{Orb}(\pi)$ be the set of permutations $w$ such that $M(w)$ is in 
 the orbit of  $M(\pi)$ under the HR-action. 
Thus, for  any $\pi\in\S_n$,  there is a unique permutation $\pi^{A}$ in $\mathrm{Orb}(\pi)$ such that all its interior vertices in 
$M(\pi^{A})$ are min 
vertices.

\begin{fact}\label{andré-in-tree}
A permutation $\pi\in\S_n$ is an Andr\'e permutation  if and only if all interior vertices of min-max tree $M(\pi)$ are min vertices.
\end{fact}
It follows that 
$\cup_{\pi\in\mathrm{And}_{n}}\mathrm{Orb(\pi)}=\S_n$, where 
$\And_n$ is the set of Andr\'e permutations 
in $\S_n$.
Let $\S_n^*$ (resp. $\mathrm{Orb}^*(\pi)$)  be the subset of permutations 
in $\S_n$ (resp. $\mathrm{Orb}(\pi)$) which have no even descents. 
By restriction on the permutations which have only odd-descents we have 
\begin{align}
\S_n^*=\cup_{\pi\in\mathrm{And}_n}
\mathrm{Orb}^*(\pi).
\end{align}

 For any subset $S\subseteq [n]$ and Andr\'e permutation $\pi$, since all interior vertices of $M(\pi)$ are min vertices, we have $\des\psi_S(M(\pi))\geq\des\pi$.

Let  $\pi\in\mathrm{And}_n$.
So all the interior vertices of  $M(\pi)$ are min vertices,
if $\pi(k)$ is a    valley  of $\pi$ with  the $\pi(k)$-factorization  $(w_1,w_2,\pi(k),w_4,w_5)$, then the position of the last letter of $w_2$ is a descent position, and the  HR action $\psi_k$ on $M(\pi)$ will shift the descent position $k-1$  to $k$, since the vertex in $M(\pi)$ corresponding to  $\pi(k)$ will be relabelled by the largest letter of its subtree, and all other vertices keep their same relative order.
Thus, the  HR action $\psi_S$ with 
$S$ being the set of indices of 
 odd-valley-positions in $\pi$ will evacuate all the even  descent
positions, and the total number of descents will remain the same, let 
 $\psi_S(\pi)=\pi'$, clearly $\pi'\in \mathrm{Orb}^*(\pi)$.
\begin{fact} For 
$\pi\in\mathrm{And}_{n}$ we have
$$
\mathrm{Orb}^*(\pi)=\mathrm{Orb}^*(\pi')=
\prod_{i\in S}(1+\psi_i)\{\pi'\},
$$
where $S$ is the set of odd ascent positions of $\pi'$.
Moreover, as $\des_0\psi_i(\pi')=\des_0(\pi')+1$, 
the folllowing identity  holds
\begin{align}
\sum_{\sigma\in 
\mathrm{Orb}^*(\pi)}p^{\des\sigma}=(1+p)^{\lfloor{n/2}\rfloor-\des\pi}p^{\des \pi}.
\end{align}
\end{fact}

Recall that $a(n,j)$  is 
the number of permutations in $\S_n$ with $j$ odd  descents and without even descents.
\begin{lem}\label{P}
 Let
$d(n,j)$ be the number of Andr\'e permutations in 
$\S_n$ with $j$ descents
for $0\leq 2j\leq n$. Then
\begin{align}\label{eq:s-d}
a(n,j)=\sum_{i=0}^j\binom{\lfloor{n/2}\rfloor-i}{j-i}d(n,i).
\end{align}
\end{lem}
\begin{proof} Applying the above  facts 
\begin{align*}
\sum_{\sigma\in \S_n^*} p^{\des(\sigma)}
&=\sum_{\pi\in \And_n} \sum_{\sigma\in \textrm{Orb}^*(\pi)}p^{\des(\sigma)}\\
&=\sum_{\pi\in \And_n} (1+p)^{\lfloor{n/2}\rfloor-\des\pi}p^{\des \pi}.
\end{align*}
We derive \eqref{eq:s-d} by extracting the coefficient of $p^j$. 
 \end{proof}

\begin{lem} \label{link:a-d}
If  $\bar d(n,i)$ is the number of min-max trees on  $[n]$ having
$i$ max interior vertices with  two children,  then 
\begin{align}\label{peak-descent}
\bar d(n,i)=\sum_{j=i}^{\lfloor\frac{n}{2}\rfloor}\binom{j}{ i}d(n,j).
\end{align}
\end{lem}
\begin{proof}
If  $\pi\in \S_n$ is an André permutation, then  
the number of interior vertices with two children of $M(\pi)$ equals $\des(\pi)$.  Any permutation $\pi\in \S_n$ such that $M(\pi)$ has 
$i$ max interior vertices with two children can be obtained 
from the André permutation $\pi^A$ in $\mathrm{Orb}(\pi)$ by choosing
$i$ interior vertices with two children 
among  the  interior vertices with two children of $M(\pi^A)$ and  then applying 
 HR operator on these $i$ vertices (to transform them into max vertices). Hence, in each orbite of an André 
  min-max tree  (i.e., the tree $M(w)$ 
 associated to an André permutation $w$) with $j$ interior 
 vertices having two children, there are 
 $\binom{j}{i}$  min-max trees on $[n]$ having 
 $i$ max interior vertices with two children. 
 The result follows by summing 
 over  all the orbits. 
	\end{proof}

Recall   that
a permutation $w$ of $[n]$ is an \emph{Andr\'e  permutation of kind II} if, for $1\leq k\leq n$, 
\begin{enumerate}
\item the subsequence of the smallest $k$ elements in $w$ has no double descent  ;
\item the subsequence of the smallest $k$ elements in $w$ ends with an ascent.
\end{enumerate}
The permutation $w$ is called \emph{Simsun}  if it satisfies 
condition  (1), \cite{CS11, FH16, PZ21}.
 For example, 
the  five   Simsun 3-permutations are:  $231, 132, 312, 123, 213$ and
the five Andr\'e 4-permutations of the second kind are: 
$1234, \, 1423,\, 3124,\, 3412,\, 4123$. 

Actually,  the number of $n$-Andr\'e permutations and that of 
$(n-1)$-simsun permutations are both equal to 
the \emph{Euler number} 
$E_n$, which  can be  defined by
$$
\sum_{n\geq 0}E_n \frac{x^n}{n!}=\sec x+\tan x.
$$
Let $D_n(x)$ (resp. $rs_n(x)$)  
be the descent polynomial of  André permutations (resp. Simsun permutations) of length $n$.  By means of generating
function argument, Chow and Shiu \cite{CS11} proved  that the descent number is equidistributed over $(n-1)$-simsun permutations
and $n$-Andr\'e permutations, i.e.,
\begin{align}\label{rs-andré}
D_n(x)=rs_{n-1}(x)=\sum_{i=0}^{n-1} d(n,i)x^i\quad (n\geq 2)
\end{align}
with $D_1(x)=1$. 

Combining Theorem~\ref{thm-typeA} and Lemma~\ref{P}
we obtain an alternative proof of the following result of   H. Sun~\cite{Sun21}.
\begin{thm}\label{thm:sun}
Let
$d(n,j)$ be the number of Andr\'e permutations in 
$\S_n$ with $j$ descents
for $0\leq 2j\leq n$ and $\bar d(n,i)$ be 
 the number of min-max trees on  $n$ vertices having
$i$ max interior vertices with  two children, then 
\begin{subequations}
\begin{align}\label{gamma-sun}
\widetilde{A}_n(x,y)&=\sum_{j=0}^{\lfloor{n/2}\rfloor}d(n,j)(x+y)^j(1+x+y+xy)^{\lfloor\frac{n}{2}\rfloor-j},\\
\overline{A}_n(x,y)&=\sum_{i=0}^{\lfloor{n/2}\rfloor}(-1)^i\bar d(n,i)(x+y)^i(1+x+y+xy)^{\lfloor{n/2}\rfloor-i},
\label{gamma-alt}
\end{align}
\end{subequations}
and
$rs_{n-1}(1+x)=\sum_{i=0}^{\lfloor{n/2}\rfloor}\bar d(n,i)x^i$, where $rs_{n-1}(x)$ is the descent polynomial of Simsun permutations.
\end{thm}
\begin{proof}
Plugging \eqref{eq:s-d} in \eqref{sym-S} we obtain 
\begin{align*}
\widetilde{A}_n(x,y)
&=\sum_{j=0}^{\lfloor{n/2}\rfloor}\sum_{i=0}^{j}
\binom{\lfloor{n/2}\rfloor-i}{j-i}d(n,i) (x+y)^{j}(1+xy)^{\lfloor\frac{n}{2}\rfloor-j}\\
&=\sum_{i=0}^{\lfloor{n/2}\rfloor}d(n,i)(x+y)^{i}
\sum_{j\geq 0}\binom{\lfloor{n/2}\rfloor-i}{j}
(x+y)^{j}(1+xy)^{\lfloor{n}{2}\rfloor-i-j}\\
&=\sum_{i=0}^{\lfloor{n/2}\rfloor}d(n,i)
(x+y)^i (1+x+y+xy)^{\lfloor{n/2}\rfloor-i},
\end{align*}
which is  the right-hand side of \eqref{gamma-sun} upon
replacing $i$ by $j$.

By \eqref{tildeA:A} and \eqref{ST}
we have $\overline{A}_n(x,y)=y^{\lfloor n/2\rfloor}\widetilde{A}_n(x,1/y)$. 
Hence  
\begin{align*}
\overline{A}_n(x,y)&=\sum_{j=0}^{\lfloor{n/2}\rfloor}d(n,j)(1+xy)^j(1+x+y+xy)^{\lfloor{n/2}\rfloor-j}.
\end{align*}
Now, rewriting $(1+xy)^j$ in the last sum as
 $$
(1+xy)^j=\sum_{i=0}^j(-1)^i\binom{j}{ i}(1+x+y+xy)^{j-i} (x+y)^i,
$$
 we obtain the right-hand side of \eqref{gamma-alt}.
\end{proof}

\begin{remark}
Comparing \eqref{gamma-eulerianpoly} with  \cite[Theorem 2]{PZ21} we notice that  $d(n,j)$ is also the number of Andr\'e permutations of kind II of $[n]$ with $j$ descents. This result is implicit in \cite{FS71, FS76, FH16}.
If $x=y$, Theorem~\ref{thm:sun} plainly reduces to the classical $\gamma$-formula of Eulerian polynomials, see \cite[Theorem~1]{PZ21},
\begin{align}\label{gamma-eulerianpoly}
A_n(x,x)=\sum_{j=0}^{\lfloor{n/2}\rfloor}d(n,j)\, 2^j\, x^j (1+x)^{n-1-2j}.
\end{align} 
Also, Lin et al. \cite[Theorem 1.1]{LMWW22} proved the $x=y$ case of~\eqref{gamma-alt}.
\end{remark}
\subsection{Gamma-positivity  of bi-Eulerian polynomials of type B}
We refine  Petersen's proof of gamma-nonnegativity of type B Eulerian polynomials in \cite{Pe15}.
 
Given a permutation $u\in\S_n$ 
we denote  by $\B(u)$  
the set of all permutations $\omega\in\B(u)$
such that  $\omega(i)=\sigma_iu(i)$ with 
$\sigma_i\in\{-,+\}$ for $1\leq i\leq n$. 
Then we have the following observations:
\begin{itemize}
\item 
if $u(i-1)<u(i)$, then $\omega(i-1)>\omega(i)$ if and only if $\sigma_i=-$,
\item
if $u(i-1)>u(i)$, then $\omega(i-1)>\omega(i)$ if and only if $\sigma_{i-1}=+$.
\end{itemize}
To put it another way, the sign $\sigma_j$ controls the descent in position $j-1$ if and only if $j-1$ is \emph{not} a descent position of $u$, and it controls the descent in position $j$ if and only if $j$ \emph{is} a descent position of $u$.

Consider the example of $u=31472865$. Then there is a descent in position $0$ if and only if $\sigma_1=-$ while there is a descent in position $1$ if and only if $\sigma_1=+$. Since $u(2)=1$ is smaller than the elements on either side of it, the sign $\sigma_2$ has no effect whatever on the descent set. With $u(3)=4$, we find that $\omega(2)>\omega(3)$ if and only if $\sigma_3=-$, but that $\sigma_3$ does not control whether $\omega(3)$ is greater than $\omega(4)$ ($\sigma_4$ does that). By considering the sign of each letter in turn.

We summarize the above  consideration more precisely in  the following 
\begin{observation}\label{fact:f}
Let $u\in\S_n$. If   $\omega\in\B_n(u)$ with $\omega(j)=\sigma_ju(j)$, then  
\begin{itemize}
\item 
If $u(j-1)<u(j)>u(j+1)$, then $\sigma_j$ controls both the descent in position $j-1$ and position $j$. That is, if $\sigma_j=+$, then in $\omega$, $j-1$ is not a descent position, but $j$ is a descent position. If $\sigma_j=-$, then in $\omega$, $j-1$ is a descent position but $j$ is not. This means, $\sigma_j$ does not change the number of descents, but it controls the parity of descent position.
\item
If $u(j-1)<u(j)<u(j+1)$, then $\sigma_j$ controls the descent on position $j-1$, but no effect on position $j$. That is, if $\sigma_j=+$, then $j-1$ is not a descent position, if $\sigma_j=-$, then $j-1$ is a descent position.
\item
If $u(j-1)>u(j)>u(j+1)$, then $\sigma_j$ controls the descent on position $j$, but no effect on position $j-1$. That is, if  $\sigma_j=+$, then $j$ is a descent position, if $\sigma_j=-$, then $j$ is not a descent position.
\item
If $u(j-1)>u(j)<u(j+1)$, then $\sigma_j$ has no effect on the descent set.
\end{itemize}
\end{observation}
The \emph{number of left peaks} of  permutation $u\in\S_n$ is  defined by
\begin{align}\label{lpk}
\mathrm{lpk}(u)=|\{1\leq i<n:u(i-1)<u(i)>u(i+1)\}|,
\end{align}
 where $u(0)=0,\,u(n+1)=n+1$. 
\begin{lem}\label{L1}
If $\omega\in\B_n$ is a permutation with $j$ odd descents and without even descents, then $|\omega|$ is a permutation  in $\S_n$ with $\mathrm{lpk}(|\omega|)\leq j$.
\end{lem}
\begin{proof}
Since $\omega$ does not have descents on even positions,  we have $\omega(1)>0$ and $\omega$ does not have double descents. Suppose $\omega(i)$ is the first valley with $\sigma_i=-$ and $\omega(k)$ is the peak closest to $\omega(i)$ on the right. Then $\omega(i)\omega(i+1)\ldots\omega(k)$ is an increasing subsequence, and there has no peak in $|\omega(i)||\omega(i+1)|\ldots|\omega(k)|$. Let $\omega_0=\omega(1)\omega(2)\ldots\omega(i-1)|\omega(i)||\omega(i+1)|\ldots|\omega(k)|\omega(k+1)\ldots\omega(n)$ then, the difference of peak sets of $\omega_0$ and $\omega$ happens on $\omega(i-1)$, $|\omega(i)|$ and $|\omega(k)|$, $\omega(k+1)$. As it is not possible that both $\omega(i-1)$ and $|\omega(i)|$ are peaks in $\omega_0$ (but $\omega(i-1)$ is a peak in $\omega$). Since $|\omega(k)|\geq\omega(k)>\omega(k+1)$, so $|\omega(k)|$ is the only possible peak candidate of $|\omega(k)|$ and $\omega(k+1)$ in $\omega_0$ ($\omega(k)$ is a peak in $\omega$). In summary, we have $\mathrm{lpk}(\omega_0)\leq\mathrm{lpk}(\omega)$. We repeat this process on $\omega_0$, finally, we obtain $\mathrm{lpk}(|\omega|)\leq\mathrm{lpk}(\omega)=j.$
\end{proof}
\begin{lem}\label{L2}
Let $g(n,i)=|\{u\in\S_n:\lpk(u)=i\}|$. Then
$$
b(n,j)=\sum_{i=0}^j\binom{\lfloor{n/2}\rfloor-i}{j-i}g(n,i) 2^i.
$$
\end{lem}
\begin{proof}
Let $u$ be a permutation in $\S_n$ with $\mathrm{lpk}(u)=i\leq j$. We can use the following process to transform it
to  a permutation of $\B_n$ with $j$ odd descent and without even descents.\\
{\sc Process A}
\begin{enumerate}
\item 
Firstly, we sign the $i$ valleys of $u$ with either $-$ or $+$, which gives $\omega_1$.
\item
Secondly, in $\omega_1$, we sign the peaks at even positions with $-$, then we obtain $\omega_2$ with all the peaks at odd positions (by Remark~\ref{fact:f}). 
\item
Thirdly, choose a $j-i$ elements subset $D$ of $C:=\{1,3,\ldots,2\lfloor\frac{n}{2}\rfloor-1\}\setminus\mathrm{LPK}(\omega_2)$, where $\mathrm{LPK}(\omega_2)$ is the position set of peaks of $\omega_2$. For $l\in D$, if $\omega_2(l)$ is a descent then we do nothing with  $\omega_2(l)$, if $\omega_2(l)$ is an ascent then we sign $\omega_2(l+1)$ (\emph{it must be a double ascent in $u$}) with $-$. For $l\notin D$ but $l\in C$, if $\omega_2(l)$ is a descent then we sign $\omega_2(l)$ (\emph{it must be a double descent in $u$})with $-$, if $\omega_2(l)$ is an ascent, then we do nothing with $\omega_2(l)$, which gives  $\omega_3$.
\item
Lastly, in $\omega_3$ we sign all the double descents at even positions with $-$, which gives $\omega_4$. 
\end{enumerate}
By Observation~\ref{fact:f}, we see that $\omega_4$ is a permutation in $\B_n$ with $j$ odd descents and without even descents.

In this process, no letter in $u$ is repeatedly signed. And we can see that for a fixed $u\in\S_n$ with $i$ peaks, by Process~A, it can produce $\binom{\lfloor{n/2}\rfloor-i}{j-i}\cdot2^i$ different permutations in $\B_n$ with $j$ odd descents and without even descents. By Lemma~\ref{L1}, for $\omega\in\B_n$ with $j$ odd descents and without even descents, we have $|\mathrm{lpk}(|\omega|)|\leq j$ and by Remark~\ref{fact:f}, the descent positions in $\omega$ are totally controlled by the signs of peaks, double descents and double ascents of $|\omega|$, that is $\omega$ can be constructed by $|\omega|$ through Process~A. This completes the proof.
\end{proof}

\begin{thm}\label{thm-typeB-gamma}
Let $g(n,j)=|\{u\in\S_n:\mathrm{lpk}(u)=j\}|$. Then
\begin{align}\label{eq:U-gamma}
\widetilde{B}_n(x,y)&=\sum_{j=0}^{\lfloor\frac{n}{2}\rfloor}g(n,j)2^j (x+y)^j(1+x+y+xy)^{\lfloor\frac{n}{2}\rfloor-j},\\
\label{eq:V-gamma}
\overline{B}_n(x,y)&=\sum_{j=0}^{\lfloor\frac{n}{2}\rfloor}(-1)^j\bar{g}(n,j)2^j (x+y)^j(1+x+y+xy)^{\lfloor\frac{n}{2}\rfloor-j}
\end{align}
with
\begin{align}\label{eq:gbar-g}
\bar{g}(n,j)
=\sum_{i=0}^{\lfloor{n/2}\rfloor-j}
\binom{i+j}{j} g(n,i+j)2^i.
\end{align}
\end{thm}
\begin{proof}
By Theorem~\ref{TB1} and Lemma~\ref{L2}, we obtain \eqref{eq:U-gamma}.
To prove \eqref{eq:V-gamma},  
by~\eqref{chebikin-polynomial-B}, \eqref{B1} and~\eqref{B2}, we first note 
$$
\overline{B}_n(x,y)=y^{\lfloor\frac{n}{2}\rfloor}\widetilde{B}_n(x,1/y).
$$
It follows from \eqref{eq:U-gamma} that
\begin{align}\label{TBB}
\overline{B}_n(x,y)=\sum_{j=0}^{\lfloor\frac{n}{2}\rfloor}g(n,j)2^j\, (1+xy)^j(1+x+y+xy)^{\lfloor\frac{n}{2}\rfloor-j}.
\end{align}
The rest of the proof is the same as that of Eq.~\eqref{gamma-alt}, so it is omitted.
\end{proof}

\begin{remark} 
When $x=y$ identity \eqref{TBB} reduces to Proposition 10 in \cite{MFMY22}.
Identity \eqref{eq:gbar-g} is equivalent to the polynomial identity:
\begin{align}
\sum_{j=0}^{\lfloor n/2\rfloor} \bar{g}(n,j) x^j
&=\sum_{j=0}^{\lfloor n/2\rfloor}\sum_{i=0}^{\lfloor\frac{n}{2}\rfloor-j}
\binom{i+j}{j} g(n,i+j)2^ix^j\nonumber \\
&=\sum_{k=0}^{\lfloor n/2\rfloor}g(n,k) (2+x)^k.
\end{align}
If  $x=y$   identity~\eqref{eq:U-gamma} reduces  to Petersen's formula for type B Eulerian polynomial $B_n(x,x)$, see \cite[Theorem 13.5]{Pe15},  
\begin{align}
B_n(x,x)=\sum_{j=0}^{\lfloor n/2\rfloor}g(n,j) \,(4x)^j(1+x)^{n-2j},
\end{align}
and Eq. \eqref{eq:V-gamma} reduces to Ma et al.'s 
formula for type B alternating descent polynomials, see \cite[Theorem 12]{MFMY22}
\begin{align}
\widehat{B}_n(x,x)=\sum_{j=0}^{\lfloor\frac{n}{2}\rfloor} \bar{g}(n,j) (-4x)^j(1+x)^{n-2j}.
\end{align}

\end{remark}
\section{Counting  permutations of type A by the parity of 
 descent positions}
 If  $\sigma=\sigma_1\cdots\sigma_n$ is a permutation in $\S_n$, the descent set $\D(\sigma)$ of $\sigma$ is 
  $\D(\sigma)=\{i:\sigma_{i}>\sigma_{i+1}\}\subseteq[n-1]$. 
  We denote by  
  $\eD(\sigma)$ (resp. $\oD(\sigma)$) the  set of even (resp. 
  odd) descents of $\sigma$. For brevity we denote their cardinalities by  $\ed(\sigma)=|\eD(\sigma)|$ and  $\od(\sigma)=|\oD(\sigma)|$. 
    
Any subset $S=\{s_1,\ldots,s_k\}_<\subseteq[n-1]$ can be encoded by   the composition $\mathrm{co}(S):=(s_1,s_2-s_1,\cdots,s_k-s_{k-1},n-s_k)$ of $n$. Clearly this correspondence is a bijection.
For  any  composition $\lambda=(\lambda_1,\ldots,\lambda_l)$ of $n$, let $S_{\lambda}$ be the subset $\{\lambda_1,\lambda_1+\lambda_2,\ldots,\lambda_1+\cdots\lambda_{l-1}\}$ of $[n-1]$ and  
define the $q$-multinomial coefficient
$$
\binom{n}{\lambda}_q:=\binom{n}{\mathrm{co}(S_\lambda)}_q=\frac{n!_q}{\lambda_1!_q\cdots\lambda_l!_q}.
$$
For   any subset $S\subseteq[n-1]$, let $\Delta_n(S):=  \{\sigma\in \S_n\mid \D(\sigma)\subseteq S\}$ and
 $R_n(S)$ be the set of rearrangements of word $1^{\lambda_1}\ldots l^{\lambda_l}$, where  $\lambda_i=s_i-s_{i-1}$ for $i\in [l]$ with  $l=k+1$, $s_0=0$ and  $s_l=n$. There is a bijection $\psi:\sigma\mapsto w$  from $\Delta_n(S)$ to  $R_n(S)$ defined by $w(j)=i$ if $\sigma(j)\in \{\sigma(s_{i-1}+1), \ldots, \sigma(s_i)\}_<$ for $j\in [n]$ and 
 $i\in [l]$. Clearly  the number of inversions of $w$, i.e., 
 $|\{i<j\mid w(i)>w(j), i,j\in [n]\}|$,  is equal to $\inv\sigma$.
By a theorem of MacMahon (see \cite[p. 41]{An88}) we obtain the following known result (see \cite[p. 227]{EC1}).
\begin{lem}\label{key0}
Let $S=\{s_1,s_2, \ldots, s_k\}_<\subseteq [n-1]$ and $\alpha_n(S,q)=\sum_{\sigma\in \Delta_n(S)}q^{\mathrm{inv}\sigma}$ . Then
$$
\alpha_n(S,q)=\binom{n}{\mathrm{co}(S)}_q.
$$
\end{lem}

To prove \eqref{pz1} we need three  more  lemmas.
For convenience,  for any subset $S\subseteq \N$ let
$S_{\rm e}=S\cap 2\N$ and $S_{\rm o}=S\cap (2\N+1)$ be the subsets of even and odd integers of $S$, respectively.
For $n\in\mathbb{N}$, let $\mathrm{O}[n]$ (resp. $\mathrm{E}[n]$) be the collection of odd (resp. even) elements of $[n]$.
Consider   the polynomial 
\begin{align}\label{typeA+}
P_n(x,y,q):=\sum_{S\subseteq [n-1]}\alpha_{n}(S,q)x^{|S_o|}y^{|S_e|}.
\end{align}

\begin{lem}\label{key1}  For $n\geq 1$ we have 
\begin{align}\label{PA-link}
A_n(x,y,q)=(1-x)^{\lfloor \frac{n}{2}\rfloor}
(1-y)^{\lfloor \frac{n-1}{2}\rfloor}P_{n}\left(\frac{x}{1-x},\frac{y}{1-y},q\right).
\end{align} 
\end{lem}
\begin{proof}
By Lemma 2.1 we have 
\begin{align*}
P_n(x,y,q)
&=\sum_{\sigma\in\S_{n}}x^{\od(\sigma)}y^{\ed(\sigma)}q^{\inv(\sigma)}
\sum_{S\subseteq [n-1]\setminus D(\sigma)}x^{|S_o|}y^{|S_e|}\\
&=\sum_{\sigma\in\S_{n}}x^{\od(\sigma)}y^{\ed(\sigma)}q^{\inv(\sigma)}
(1+x)^{\lfloor \frac{n}{2}\rfloor-\od(\sigma)}(1+y)^{\lfloor \frac{n-1}{2}\rfloor-\ed(\sigma)}
\end{align*}
as there are $\lfloor \frac{n}{2}\rfloor-\od(\sigma)$ odd (resp. $\lfloor \frac{n-1}{2}\rfloor-\ed(\sigma)$ even) integers   in $[n-1]\setminus D(\sigma)$.
In other words, we can write $P_n(x,y,q)$ as 
\[
P_n(x,y,q)= (1+x)^{\lfloor \frac{n}{2}\rfloor}
(1+y)^{\lfloor \frac{n-1}{2}\rfloor}A_{n}\left(\frac{x}{1+x},\frac{y}{1+y},q\right),
\]
which is equivalent to \eqref{PA-link}.
\end{proof}

\begin{remark} Let $P_n(x)=\sum_{S\subseteq [n-1]}\alpha_n(S,1) x^{|S|}$. It is not diffucult to see that 
$$
P_n(x)=\sum_{k=0}^{n-1} (k+1)! S(n,k+1) x^k,
$$
where $S(n,k)$ denotes the  Stirling number of the second kind, i.e., 
the number of ways to partition a set of $n$ objects into  $k$  non-empty subsets (see \cite{EC1}). So, when $x=y$,
formula \eqref{PA-link} reduces to   the \emph{Frobenius formula}, see \cite{FS70},
\begin{align}
A_n(x)=\sum_{k=1}^n k!S(n,k) x^{k-1} (1-x)^{n-k}.
\end{align}
\end{remark}

\begin{lem}\label{key2} We have
\begin{align}
B(t, x):&=\sum_{n\geq1}P_{2n}(x,0,q)\frac{t^{2n}}{(2n)!_q}=\frac{(\cosh_q t-1)(1-x(\cosh_q t-1))+x\sinh_q^2 t}{1-x(\cosh_q t-1)},\label{eq:B}\\
C(t,x):&=\sum_{n\geq1}P_{2n-1}(x,0,q)\frac{t^{2n-1}}{(2n-1)!_q}=\frac{\sinh_q t}{1-x(\cosh_q t-1)}.\label{eq:C}
\end{align}
\end{lem}
\begin{proof} There is a bijection between  the set of compositions $\gamma=(\gamma_1,\cdots,\gamma_l)$ of $2n$  such that $\gamma_1,\,\gamma_1+\gamma_2,\,\ldots,\,\gamma_1+\gamma_2+\cdots+\gamma_{l-1}$ are odd numbers and the set of subsets $S_\gamma$ of
$\mathrm{O}[2n]$. Hence
\begin{align*}
\sum_{n\geq1}P_{2n}(x,0,q)\frac{t^{2n}}{(2n)!_q}
&=\sum_{n\geq1}\left(\sum_{S\subseteq \mathrm{O}[2n]}\alpha_{2n}(S,q)
x^{|S|}\right)\frac{t^{2n}}{(2n)!_q}\\
&=\sum_{l\geq1}\left(\sum_{\gamma}\frac{t^{\gamma_1}}{\gamma_1!_q}\cdots\frac{t^{\gamma_l}}{\gamma_l!_q}\right)x^{l-1}\\
&=\sum_{i\geq1}\frac{t^{2i}}{2i!_q}+x\sum_{l\geq2}\left(\sum_{i\geq1}\frac{t^{2i-1}}{(2i-1)!_q} \right)^2\left(x\sum_{i\geq1}\frac{t^{2i}}{2i!_q}\right)^{l-2}\\
&=\cosh_q t-1
+\frac{x\sinh_q^2 t}{1-x(\cosh_q t-{1})},
\end{align*}
which gives \eqref{eq:B}.

In the same vein, wa have 
\begin{align*}
\sum_{n\geq1}P_{2n-1}(x,0,q)\frac{t^{2n-1}}{(2n-1)!_q}&=\sum_{n\geq 1}\sum_{S\subseteq \mathrm{O}[2n-1]}\alpha_{2n-1}(S,q)x^{|S|}\frac{t^{2n-1}}{(2n-1)!_q}\\
&=\sum_{l\geq1}\left(\sum_{\gamma}\frac{t^{\gamma_1}}{\gamma_1!_q}\cdots\frac{t^{\gamma_l}}{\gamma_l!_q}x^{l-1}\right)\\
&=\sum_{l\geq1}\left(\sum_{i\geq1}\frac{t^{2i-1}}{(2i-1)!_q} \right)\left(x\sum_{i\geq1}\frac{t^{2i}}{(2i)!_q}\right)^{l-1},
\end{align*}
which is clearly equal to \eqref{eq:C}.
\end{proof}
Next we generalize 
 \eqref{eq:B} and \eqref{eq:C} to  the general  $y$.
\begin{lem}\label{key3} We have 
\begin{align}
\sum_{n\geq1}P_{2n}(x,y,q)\frac{t^{2n}}{(2n)!_q}&=
\frac{B(t,x)}{1-yB(t,x)},\label{p12}\\
\sum_{n\geq1}P_{2n-1}(x,y,q)\frac{t^{2n-1}}{(2n-1)!_q}
&=\frac{C(t,x)}{1-yB(t,x)}.\label{p13}
\end{align}
\end{lem}
\begin{proof} Consider
$$
P_n(x,y,q)
=\sum_{(\sigma, S)}x^{|S_o|}y^{|S_e|} q^{\inv \sigma}\qquad 
(\sigma\in \S_n\; \textrm{and}\; D(\sigma)\subseteq S\subseteq [n-1]).
$$
There is a bijection between the set of subsets $S$ of $[n-1]$ with fixed even integers 
$S_{\rm e}=\{m_1<\cdots< m_{l-1}\}\subset \mathrm{E}[n-1]$ and the set of sequences of compositions of $m_i-m_{i-1}$ with odd parts 
for $i\in [l]$
with  $m_0=0$ and $m_l=n$.  Let $\co(S_{\rm e})=(n_1, \ldots, n_l)$ be the corresponding composition of $n$. Then
\begin{align*}
\sum_{n\geq1}P_{2n}(x,y,q)\frac{t^{2n}}{(2n)!_q}&=\sum_{l\geq 1}
\prod_{i=1}^{l-1}
\left[\sum_{S_i\subseteq \mathrm{O}[2n_i]}
\alpha_{2n_i}(S_i,q)  x^{|S_i|}
\frac{t^{2n_i}}{(2n_i)!_q}y \right]\\
&\hspace{2cm}\times \left[\sum_{S_l\subseteq \mathrm{O}[2n_l]}\alpha_{2n_l}(S_l,q)x^{|S_l|}\frac{t^{2n_l}}{(2n_l)!_q}\right],
\end{align*}
which is equal to  $\sum_{l\geq1}y^{l-1}\cdot B(t,x)^l=\frac{B(t,x)}{1-yB(t,x)}$.

Similary, we have 
\begin{align*}
\sum_{n\geq1} P_{2n-1}(x,y,q)\frac{t^{2n-1}}{(2n-1)!_q}
&=\sum_{l\geq1}\prod_{i=1}^{l-1}
\left(\sum_{S_i\subseteq \mathrm{O}[2n_i]}\alpha_{2n_i}(S_i,q)
x^{|S_i|}\frac{t^{2n_i}}{(2n_i)!}y \right)\\
&\hspace{1cm}\times \left(\sum_{S_l\subseteq \mathrm{O}[2n_{l}-1]}\alpha_{2n_{l}-1}(S_{l},q)x^{|S_{l}|}\frac{t^{2n_{l}-1}}{(2n_{l}-1)!_q}\right),
\end{align*}
which  can be written as  $\sum_{l\geq1}y^{l-1}\cdot B(t,x)^{l-1}\cdot C(t,x)=\frac{C(t,x)}{1-yB(t,x)}$.
\end{proof}

We  obtain \eqref{pz1} by combining  Lemma~\ref{key1}, Lemma~\ref{key2}  and Lemma~\ref{key3}.
\section{Counting  permutations of type B by the parity of  descent positions}
Let $\B_n^+$ (resp. $\B_n^-$) be the subset of permutations in 
 $\B_n$ whose first entry is positive (resp. negative). 
Clearly the doubleton  $\{\mathcal{B}_n^-, \mathcal{B}_n^+\}$ is a partition of $\B_n$. Introduce the corresponding enumerative polynomials:
\begin{align*}
{B}_n^-(x,y)=\sum_{\sigma\in \B_n^-} x^{\odes\sigma}y^{\edes\sigma},\quad {B}_n^+(x,y)=\sum_{\sigma\in\mathcal{B}_n^+}x^{\odes\sigma}y^{\edes\sigma}.
\end{align*}
Then $
{B}_n(x,y)={B}_n^-(x,y)+{B}_n^+(x,y)$.


For $\tau\in\B_n$, let $\tau^-$ be the permutation in $\B_n$
such that  $\tau^-(i)=-\tau(i)$ for $i\in [n]$. 
It is clear that the mapping $\rho: \tau\longmapsto\tau^-$ is an involution on $\mathcal{B}_n$ such that 
\begin{align}
\begin{split}
 \odes\tau+\odes\tau^-&=\lfloor n/2\rfloor,\\
 \edes\tau+\edes\tau^-&=\lfloor (n+1)/2\rfloor.
\end{split}
\end{align}
Besides,  the restriction of $\rho$ on $\B^+_n$ sets up a bijection
 $\rho: \B^+_n\to \B^-_n$, 
therefore
\begin{align}
\label{linkBQeven}
B_{n}^-(x,y)&=x^{\lfloor n/2\rfloor}y^{\lfloor (n+1)/2\rfloor} 
B_{n}^+\left(1/x,1/y\right).
\end{align}
So, we need only to compute the exponential generating functions of $B_{n}^+(x,y)$.

For $\sigma\in B_n$, we denote by $D(\sigma)$  the set of descents of $\sigma$.  If $S$ is  a  subset of $[n-1]$  let $\alpha^+_n(S)$ be the number of permutations $\sigma\in \B^+_n$ such that  
$D(\sigma)\subseteq S$.
 A set composition of set $\Omega$  is an $\ell$-tuple $(\Omega_1, \ldots, \Omega_\ell)$ of subsets of   $\Omega$ such that  
$\{\Omega_1, \ldots, \Omega_\ell\}$ is a set partition of  $\Omega$.

\begin{lem} Let $S=\{s_1<\cdots<s_k\}\subseteq[n-1]$, and $s_0=0$ and $s_{k+1}=n$. Then
\begin{align}
\alpha^+_n(S)=\binom{n}{\mathrm{co}(S)} 2^{n-s_1}.
\end{align}
\end{lem}
\begin{proof}
We can construct the  permutations $\sigma\in\B^+_n$ with  
  $D(\sigma)\subseteq S$ as in the following:
  \begin{itemize}
\item   partition $[n]$ to obtain a set-composition $(\Omega_1, \ldots, \Omega_{k+1})$ of $[n]$ with $|\Omega_i|=s_i$ for $1\leq i\leq k$ and $|\Omega_{k+1}|=n-s_k$,
 \item  sign the elements in $\Omega_i$ by $\epsilon\in \{-1, 1\}$ for $i=2,\ldots k+1$. 
 \item arrange the elements in each block $\Omega_i$ increasingly.
 \end{itemize} 
  It is clear that the number of such permutations is
%
%
$$
 \binom{n}{s_1-s_0,s_2-s_1,\ldots,s_{k+1}-s_k}\, 2^{n-s_1}.
$$
This is the desired formula.
\end{proof}
Similar to permutations of type A (see \eqref{typeA+}), consider the polynomial
\begin{align}\label{typeB+}
Q^+_n(x,y)=\sum_{S\subseteq [n]}\alpha^+_{n}(S)x^{|S_o|}y^{|S_e|}.
\end{align}

\begin{lem}\label{lem:BQlink} We have 
\begin{align}
B_{2n}^+(x,y)&=(1-x)^n(1-y)^{n-1}Q^+_{2n}\left(\frac{x}{1-x}, 
\frac{y}{1-y}\right),\label{evenQ+}\\
B_{2n-1}^+(x,y)&=(1-x)^{n-1}(1-y)^{n-1}Q^+_{2n-1}\left(\frac{x}{1-x}, 
\frac{y}{1-y}\right).\label{oddQ+}
\end{align}
\end{lem}
\begin{proof}
For even index we have
\begin{align}
Q^+_{2n}(x,y)
=\sum_{\sigma\in\B_{2n}^+}
\sum_{S\subseteq [2n]}
\sum_{\substack{\eD(\sigma)\subseteq S_e\\ \oD(\sigma)\subseteq S_o}}x^{|S_o|}y^{|S_e|}\label{eq:evenQ+}.
\end{align}
Now, for any fixed $\sigma\in\B_{2n}^+$,  writing
$T_0=S_e\setminus \eD(\sigma)$ and $T_1=S_e\setminus \oD(\sigma)$, then 
$|S_e|= \ed(\sigma)+|T_0|$ and $|S_o|= \od(\sigma)+|T_1|$; hence
the inner double sum at the right-hand side of~\eqref{eq:evenQ+} 
is a sum over the pairs $(T_0, T_1)$ such that
$T_0\subseteq \mathrm{E}[2n]$ and $T_1\subseteq \mathrm{O}[2n]$, 
and thus  equal to
\begin{align}
y^{\ed(\sigma)}x^{\od(\sigma)}(1+y)^{n-1-\ed(\sigma)}(1+x)^{n-\od(\sigma)}.\label{bp2}
\end{align}
Therefore
\begin{align}
Q_{2n}^+(x,y)=
(1+x)^n(1+y)^{n-1}B_{2n}^+\left(\frac{x}{1+x},\frac{y}{1+y}\right),
\end{align}
which is equivelent to \eqref{evenQ+}.

For odd index, a similar reasoning can be applied with regard to the sum
\begin{align}
Q_{2n-1}^+(x,y)&=\sum_{\sigma\in\B_{2n-1}^+}
\sum_{S\subseteq [2n-1]}
\sum_{\substack{\eD(\sigma)\subseteq S_e\\ \oD(\sigma)\subseteq S_o}}x^{|S_o|}y^{|S_e|}\label{bp1}
\end{align}
and leads to the formula
\begin{align}
Q_{2n-1}^+(x,y)=
(1+x)^{n-1}(1+y)^{n-1}B_{2n-1}^+\left(\frac{x}{1+x},\frac{y}{1+y}\right).\label{obp3}
\end{align}
which is equivalent to \eqref{oddQ+}.
\end{proof}

\begin{lem}
We have 
\begin{align}\label{eq:G}
G:=\sum_{n\geq 1} Q^+_{2n}(x,0)\frac{t^{2n}}{(2n)!}&=\cosh(t)-1+
\frac{x\sinh(t)\sinh(2t)}{1-x(\cosh (2t)-1)},\\
\intertext{and} \label{eq:H}
H:=\sum_{n\geq1}\sum_{S\subseteq \mathrm{O}[2n]}
\binom{2n}{\co(S)}2^{2n}x^{|S|}\frac{t^{2n}}{(2n)!}&=\cosh(2t)-1+
\frac{x\sinh^2(2t)}{1-x(\cosh (2t)-1)}.
\end{align}
\end{lem}
\begin{proof} By definition, if
$S=\{s_1, s_2,\ldots, s_{l-1}\}_<\subseteq \mathrm{O}[2n]$, let
$\gamma_1=s_1$, $\gamma_i=s_i-s_{i-1}$ for $i=2,\ldots, l$ with $s_l=2n-1$, then $\gamma_1$ is odd and $\gamma_i$ are even for $i=2,\ldots, l$. Therefore
\begin{align}
G&=\sum_{n\geq1}\sum_{S\subseteq \mathrm{O}[2n]}
\binom{2n}{\co(S)}_{B^+}x^{|S|}\frac{t^{2n}}{(2n)!}\nonumber\\
\label{bp6}
&=\sum_{n\geq1}\left(\sum_{\gamma}\frac{1}{\gamma_1!}\cdots\frac{1}{\gamma_l!}x^{l-1}\right)2^{2n-\gamma_1}t^{2n}\\
&=\sum_{i\geq1}\frac{t^{2i}}{(2i)!}+\sum_{l\geq2}\left(\sum_{i\geq1}\frac{(t)^{2i-1}}{(2i-1)!} \right)\left(x\sum_{i\geq1}\frac{(2t)^{2i-1}}{(2i-1)!} 
\right)\left(x\sum_{i\geq1}\frac{(2t)^{2i}}{(2i)!}\right)^{l-2}\nonumber\\
&=\sum_{i\geq1}\frac{t^{2i}}{(2i)!}+x\left(\sum_{i\geq1}\frac{t^{2i-1}}{(2i-1)!}\right)\left(\sum_{i\geq1}\frac{(2t)^{2i-1}}{(2i-1)!}\right)\frac{1}{1-x\sum_{i\geq1}\frac{(2t)^{2i}}{(2i)!}},\nonumber
\end{align}
which is the right-hand side of \eqref{eq:G}. 
Next,
\begin{align}
H
\label{bp7}
&=\sum_{n\geq1}\left(\sum_{\gamma}\frac{1}{\gamma_1!}\cdots\frac{1}{\gamma_l!}x^{l-1}\right)(2t)^{2n}\\
&=\sum_{i\geq1}\frac{(2t)^{2i}}{(2i)!}+\frac{1}{x}\sum_{l\geq2}\left(x\sum_{i\geq1}\frac{(2t)^{2i-1}}{(2i-1)!} \right)^2\left(x\sum_{i\geq1}\frac{(2t)^{2i}}{(2i)!}\right)^{l-2}\nonumber\\
&=\sum_{i\geq1}\frac{(2t)^{2i}}{(2i)!}+x\left(\sum_{i\geq1}\frac{(2t)^{2i-1}}{(2i-1)!}\right)^2\frac{1}{1-x\sum_{i\geq1}\frac{(2t)^{2i}}{(2i)!}},\nonumber
\end{align}
which is the right-hand of \eqref{eq:H}.
\end{proof}

\begin{lem} We have
\begin{align}
F&:=\sum_{n\geq1}\left(\sum_{S\subseteq \mathrm{O}[2n-1]}
\binom{2n-1}{\co(S)}x^{|S|}\right)2^{2n-1}\frac{t^{2n-1}}{(2n-1)!}=\frac{\sinh(2t)}{1-x(\cosh(2t)-1)},
\label{def:F}\\
L&:=\sum_{n\geq1}\left(\sum_{S\subseteq 
\mathrm{O}[2n-1]}\binom{2n-1}{\co(S)}_B x^{|S|}\right)\frac{t^{2n-1}}{(2n-1)!}=\frac{\sinh(t)}{1-x(\cosh(2t)-1)}.
\label{def:Fbar}
\end{align}
\end{lem}
\begin{proof} By definition, if
$S=\{s_1, s_2, \ldots, s_{l-1}\}_<\subseteq \mathrm{O}[2n-1]$, 
let
$\gamma_1=s_1$, $\gamma_i=s_i-s_{i-1}$  with $s_l=2n-1$, then $\gamma_1$ is odd and $\gamma_i$ are even for $i=2,\ldots, l$. Therefore
\begin{align*}
F&=\sum_{n\geq1}\left(\sum_{\gamma}\frac{1}{\gamma_1!}\cdots\frac{1}{\gamma_l!}x^{l-1}\right)(2t)^{2n-1}\\
&=\sum_{l\geq1}\left(\sum_{i\geq1}\frac{(2t)^{2i-1}}
{(2i-1)!} \right)
\left(x\sum_{i\geq1}\frac{(2t)^{2i}}{(2i)!}\right)^{l-1}\\
&=\left(\sum_{i\geq1}\frac{(2t)^{2i-1}}{(2i-1)!}\right)\frac{1}{1-x\sum_{i\geq1}\frac{(2t)^{(2i)}}{(2i)!}},\nonumber
\end{align*}
which equals  the right-hand  side of \eqref{def:F}, besides
\begin{align*}
L
&=\sum_{n\geq1}\left(\sum_{\gamma}\frac{1}{\gamma_1!}\cdots\frac{1}{\gamma_l!} x^{l-1}\right) 2^{2n-1-\gamma_1}t^{2n-1}\\
&=\sum_{l\geq1}
\left(\sum_{i\geq1}\frac{t^{2i-1}}{(2i-1)!} \right)\left(x\sum_{i\geq1}\frac{(2t)^{2i}}{(2i)!}\right)^{l-1}\nonumber\\
&=\left(\sum_{i\geq1}\frac{t^{2i-1}}{(2i-1)!}\right)
\frac{1}{1-x\sum_{i\geq1}\frac{(2t)^{2i}}{(2i)!}},
\end{align*}
which is  equal to the right-hand side of \eqref{def:Fbar}.
\end{proof}

\begin{lem}\label{gfQ+} We have 
\begin{align}
\sum_{n\geq 1}Q_{2n}^+(x,y)\frac{t^{2n}}{(2n)!}&=\frac{G}{1-yH},\label{gf:Qeven}\\
\sum_{n\geq 1}Q_{2n-1}^+(x,y)\frac{t^{2n-1}}{(2n-1)!}&=L+ \frac{yF\,G}{1-yH}.\label{gf:Qodd}
\end{align}
\end{lem}
\begin{proof}
The left hand side of~\eqref{gf:Qeven} is
\begin{align*}
&\sum_{n\geq1}\left(\sum_{S\subseteq [2n]}\alpha_{2n}^+(S)y^{|S_e|}x^{|S_o|}\right)\frac{t^{2n}}{2n!}\\
&=\sum_{n\geq1}\left[\sum_{S_1\subseteq \mathrm{O}[2m_1]}\binom{2m_1}{\co(S_1)}_Bx^{|S_1|}\frac{t^{2m_1}}{2m_1!}y \right]\cdot
\prod_{i=2}^l\left[\sum_{S_i\subseteq \mathrm{O}[2m_i]}
\binom{2m_i}{\co(S_i)}2^{2m_i}x^{|S_i|}\frac{t^{2m_i}}{2m_i!}y \right]\\
&=\sum_{l\geq1}y^{l-1}\cdot G\cdot H^{l-1},
\end{align*}
which equals   $\frac{G}{1-yH}$.
The left-hand side of~\eqref{gf:Qodd} is
\begin{align*}
&\sum_{n\geq1}\left[\sum_{S_1\subseteq \mathrm{O}[2m_1]}\binom{2m_1}{\co(S_1)}_Bx^{|S_1|}\frac{t^{2m_1}}{2m_1!}y \right]\cdot
\prod_{i=2}^l
\left[\sum_{S_i\subseteq \mathrm{O}[2m_i]}
\binom{2m_i}{\co(S_2)}2^{2m_i}x^{|S_i|}\frac{t^{2m_i}}{2m_i!}y \right]\\
&=L+yF\cdot G\sum_{l\geq0}(yH)^l,
\end{align*}
which equals $L+\frac{yF\cdot G}{1-yH}$.
\end{proof}

Now, combining Lemma~\ref{lem:BQlink} and Lemma~\ref{gfQ+} we have
\begin{align}\label{bp12}
\sum_{n\geq 1}B_{2n}^+(x,y)\frac{t^{2n}}{(2n)!}&=\frac{(\cosh(at)-1)(2x\cosh(at)+x+1)}{1+xy-(x+y)\cosh(2at)},\\\label{obp13}
\sum_{n\geq 1}B_{2n-1}^+(x,y)\frac{t^{2n-1}}{(2n-1)!}&=\frac{\sinh(at)(x-1)(2\cosh(at)y-y-1)}{a(xy+1-(x+y)\cosh(2at))}
\end{align}
with  $a^2=(1-x)(1-y)$. 
It follows from \eqref{linkBQeven} that 
\begin{gather}\label{bp11}
\sum_{n\geq 1}B_{2n}^-(x,y)\frac{t^{2n}}{(2n)!}=\frac{y(\cosh(at)-1)(2\cosh(at)+x+1)}{1+xy-(x+y)\cosh(2at)},\\
\label{obp14}
\sum_{n\geq 1}B_{2n-1}^-(x,y)\frac{t^{2n-1}}{(2n-1)!}=\frac{y\sinh(at)(x-1)(-2\cosh(at)+y+1)}{a(xy+1-(x+y)\cosh(2at))}.
\end{gather}
Combining~\eqref{bp12} with \eqref{bp11} and 
\eqref{obp13} with \eqref{obp14}, we complete the proof of 
Theorem~\ref{thmB1}.

\section{Concluding remarks}    
In  \cite{CS73}  Carlitz and Scoville also considered the more general 
 modulus  $m>2$ for descents  rather 
than parity, i.e., $m=2$. They obtained a general generating function.  However, apart from $m = 2$ 
the generating function is quite explicit only 
 for certain special
cases when $m = 4$. For the $q$-analogue, there are some nice generating functions given by
Kur\c{s}ung\"{o}z and Yee~\cite{KY11}.
It would be very interesting to have results in this direction.

\small

\section*{Acknowledgement}
The authors thank Yao Dong and Mingjian Ding for correcting 
some flaws  in  initial versions of the manuscript.
They  should like to thank the two referees 
for thier  careful reading of the manuscript and  suggestions. The first author's work was supported by the National Natural Science Foundation of China grant 12201468.



\begin{thebibliography}{99}
\bibitem{An79} D.  André, Développements de $\sec x$ et de $\tan x$. \textbf{C.R. Acad. Sci. Paris} 88, 965–967 (1879)
\bibitem{An88} \textsc{G. Andrews}, 
\textit{The theory of partitions}. 
Reprint of the 1976 original. \textbf{Cambridge Mathematical Library. Cambridge University Press, Cambridge}, 1998.
\bibitem{BB05}
\textsc{A. Bj\"orner and  F. Brenti}, \textit{Combinatorics of Coxeter groups}. \textbf{Graduate Texts in Mathematics, 231. Springer, New York}, 2005. xiv+363 pp.
\bibitem{CS73} \textsc{L. Carlitz and R. Scoville},  \textit{Enumeration of rises and falls by position}. \textbf{Discrete Math.} 5 (1973), 45--59.
\bibitem{Che08} \textsc{D. Chebikin}, \textit{Variations on Descents and  Inversions in Permutations}, 
\textbf{Electron. J. Combin.} 15 (2008), \#R132.
\bibitem{CF22}
\textsc{W. Y. C. Chen and  A. M. Fu}, 
\textit{A context-free grammar for the e-positivity of the trivariate second-order Eulerian polynomials}. 
\textbf{Discrete Math.} 345 (2022), no. 1, Paper No. 112661, 9 pp.
\bibitem{CG07}
\textsc{C.-O. Chow and I. M. Gessel}, \textit{On the descent numbers and major indices for the hyperoctahedral group}, \textbf{Adv. Appl. Math.} 38, No. 3, 275--301 (2007).
\bibitem{CS11}
\textsc{C.-O. Chow and W. C. Shiu}, \textit{Counting Simsun permutations by descents}, \textbf{Ann. Comb.} 15, 625-635 (2011).
\bibitem{DZ21} \textsc{M. J. Ding and B. X. Zhu}, \textit{Stability of combinatorial polynomials and its applications}, arXiv:2106.12176 [math.CO].
\bibitem{FH16}
\textsc{D. Foata and G.-N. Han},  
\textit{André permutation calculus: a twin Seidel matrix sequence}. 
\textbf{Sém. Lothar. Combin.} 73 (2016), Art. B73e, 54 pp.
\bibitem{FS70}
\textsc{D. Foata and M. P. Sch\"utzenberger}, \textit{Th\'eorie g\'eom\'etrique des polyn\^omes eul\'eriens}. \textbf{Lecture Notes in Mathematics, Vol. 138 Springe-Verlag, Berlin-New York} 1970 v+94 pp.
\bibitem{FS71}
\textsc{D. Foata and M. P. Sch\"utzenberger}, \textit{Nombres d'Euler et permutations alternantes}, Manuscript, University of Florida, Gainesville, FL, 1971, available at 
\url{https://irma.math.unistra.fr/~foata/paper/pub18.pdf}.
\bibitem{FS76}
\textsc{D. Foata and V. Strehl}, \textit{Euler numbers and variations of permutations}, \textbf{in: Atti dei Convegni Lincei}, vol. 17, Tomo I, 1976, pp. 119-131. 
\bibitem{Ga79} \textsc{A. M. Garsia}, 
\textit{On the "maj" and "inv" $q$-analogues of Eulerian polynomials},
\textbf{Linear and Multilinear Algebra} 8 (1979/80), no. 1, 21--34.
\bibitem{GZ14}\textsc{I. M. Gessel and Y. Zhuang},
\textit{Counting permutations by alternating descents},  
\textbf{Electron. J. Combin}. 21 (4), 2014, Paper \#P4.23
\bibitem{GJ04} \textsc{Ian P. Goulden and David M. Jackson}, 
\textit{Combinatorial enumeration.
With a foreword by Gian-Carlo Rota}. Reprint of the 1983 original. \textbf{Dover Publications, Inc., Mineola, NY}, 2004. xxvi+569 pp. 
\bibitem{HR98}
\textsc{G. Hetyei and  E.  Reiner}, 
\textit{Permutation trees and variation statistics}. 
\textbf{European J. Combin}. 19 (1998), no. 7, 847--866.
\bibitem{JV15} \textsc{M. Josuat-Vergès},
\textit{A generalization of Euler numbers to finite Coxeter groups}. 
\textbf{Ann. Comb.} 19 (2015), no. 2, 325--336.


\bibitem{KY11}
\textsc{Kur\c{s}ung\"{o}z, Ka\u{g}an and Yee, Ae Ja},
\textit{Alternating permutations and the mth descents}.
\textbf{Discrete Math.} 311, No. 22, 2610-2622 (2011).
\bibitem{LMWW22}
 \textsc{Z.-C. Lin, S.-M. Ma,  D. G. L. Wang and  L. Wang},
 \textit{Positivity and divisibility of enumerators of alternating descents}. \textbf{Ramanujan J.} 58 (2022), no. 1, 203--228.
\bibitem{MY16}
 \textsc{S.-M. Ma and  Y.-N. Yeh}, \textit{Enumeration of permutations by number of alternating descents}. \textbf{Discrete Math.} 339 (2016), no. 4, 1362--1367. 
\bibitem{MMY20}
\textsc{S.-M. Ma, J. Ma, Y.-N.  Yeh},
\textit{David-Barton type identities and alternating run polynomials}. 
\textbf{Adv. in Appl. Math.} 114 (2020), 101978, 19 pp.
\bibitem{MFMY22}
\textsc{S.-M. Ma,   Q. Fang, T. Mansour and  Y.-N. Yeh},  \textit{Alternating Eulerian polynomials and left peak polynomials}. \textbf{Discrete Math.} 345 (2022), no. 3, Paper No. 112714, 12 pp.
\bibitem{oeis}   OEIS Foundation Inc. (2022), The On-Line Encyclopedia of Integer Sequences, Published electronically at http://oeis.org 
\bibitem{Pan22} \textsc{Q. Q.  Pan}, \textit{A new combinatorial formula for alternating descent polynomials}, arXiv preprint arXiv:2207.06212 [math.CO].
\bibitem{PZ19}
\textsc{Q. Q. Pan and J. Zeng}, \textit{A q-analogue of generalized Eulerian polynomials with applications}. \textbf{Adv. in Appl. Math.} 104 (2019), 85--99.
\bibitem{PZ21}
\textsc{Q. Q. Pan and J. Zeng}, \textit{Br\"anden's (p,q)-Eulerian polynomials, Andr\'e permutations and continued fractions}, \textbf{J. Combin. Theory Ser. A} 181 (2021) 105445.
\bibitem{Pe15}
\textsc{T. K. Petersen},  
\textit{Eulerian numbers.
With a foreword by Richard Stanley}. \textbf{Birkhäuser Advanced Texts: Basler Lehrbücher. [Birkhäuser Advanced Texts: Basel Textbooks] Birkhäuser/Springer, New York}, 2015. xviii+456 pp.
\bibitem{Re12}  
\textsc{J. B. Remmel},  
\textit{Generating functions for alternating descents and alternating major index}. 
\textbf{Ann. Comb.} 16 (2012), no. 3, 625--650.
\bibitem{St76}
\textsc{R.  P. Stanley}, 
\textit{Binomial posets, Möbius inversion, and permutation enumeration}.
\textbf{J. Combinatorial Theory Ser. A} 20 (1976), no. 3, 336--356.
\bibitem{EC1} \textsc{R.P. Stanley}, \textit{Enumerative Combinatorics}, vol. 1, second ed., in: Cambridge Studies in Advanced Mathematics, vol. 49, \textbf{Cambridge University Press}, Cambridge, 2012.
\bibitem{Sun18} \textsc{H.  Sun}, \textit{A New Class of Refined Eulerian Polynomials}, \textbf{J. Integer Seq.} Vol. 21 (2018), Article 18.5.5.
\bibitem{Sun21} \textsc{H.  Sun}, \textit{The $\gamma$-positivity of bivariate Eulerian polynomials via the Hetyei-Reiner action}. \textbf{European J. Combin.} 92 (2021), Paper No. 103166, 8 pp.
\bibitem{SZ19} \textsc{Y.  Sun and  L. Zhai}, 
\textit{Some properties of a class of refined Eulerian polynomials}. 
\textbf{J. Math. Res. Appl.} 39 (2019), no. 6, 593--602.                       
\end{thebibliography}
\end{document}